\documentclass[10pt,reqno,twoside,notitlepage,12pt]{amsart}%
\usepackage{amsmath}
\usepackage{amsfonts}
\usepackage{amssymb}
\usepackage{graphicx}%
\providecommand{\U}[1]{\protect\rule{.1in}{.1in}}
\newtheorem{theorem}{Theorem}

\newtheorem{algorithm}[theorem]{Algorithm}

\newtheorem{conclusion}[theorem]{Conclusion}

\newtheorem{definition}[theorem]{Definition}
\newtheorem{example}[theorem]{Example}
\newtheorem{examples}[theorem]{Examples}

\newtheorem{proposition}[theorem]{Proposition}
\newtheorem{remark}[theorem]{Remark}

\newenvironment{Exmps}{\begin{examples}\em}{\end{examples}}
\newtheorem{prf}{Proof}

\begin{document}
\title[\textbf{On} \textbf{Limits of Eventual Families of Sets}]{\textbf{Limits of Eventual Families of Sets with Application to Algorithms for
the Common Fixed Point Problem}}
\author{Yair Censor and Eliahu Levy}
\address{{Yair Censor, Department of Mathematics, University of Haifa, Mt. Carmel,
Haifa 3498838, Israel. Email: yair@math.haifa.ac.il.}\\
{Eliahu Levy, Department of Mathematics, The Technion--Israel Institute of
Technology, Technion City, Haifa 32000, Israel. Email:
eliahu@math.technion.ac.il.}}
\date{August 14, 2021. Revised: March 4, 2022.}

\begin{abstract}
We present an abstract framework for asymptotic analysis of convergence based
on the notions of eventual families of sets that we define. A family of
subsets of a given set is called here an \textquotedblleft eventual
family\textquotedblright\ if it is upper hereditary with respect to inclusion.
We define accumulation points of eventual families in a Hausdorff Topological
space and define the \textquotedblleft image family\textquotedblright\ of an
eventual family. Focusing on eventual families in the set of the integers
enables us to talk about sequences of points.

We expand our work to the notion of a \textquotedblleft
multiset\textquotedblright\ which is a modification of the concept of a set
that allows for multiple instances of its elements and enable the development
of \textquotedblleft multifamilies\textquotedblright\ which are either
\textquotedblleft increasing\textquotedblright\ or \textquotedblleft
decreasing\textquotedblright. The abstract structure created here is motivated
by, and feeds back to, our look at the convergence analysis of an iterative
process for asymptotically finding a common fixed point of a family of operators.

\end{abstract}
\maketitle

\section{Introduction\label{sect:introduction}}

In this paper we present an abstract framework for asymptotic analysis of
convergence based on the notions of eventual families of sets that we define.
A family $\mathcal{F}$ of subsets of a set $X$ is called here an
\textquotedblleft eventual family\textquotedblright\ if $S\in\mathcal{F}$ and
$S^{\prime}\supseteq S\,$\ implies $S^{\prime}\in\mathcal{F},$ i.e., if it is
upper hereditary with respect to inclusion. If $S\in\mathcal{F}$ and
$S^{\prime}\subseteq S\,$\ implies $S^{\prime}\in\mathcal{F},$ i.e., if it is
lower hereditary with respect to inclusion, then we call it a
\textquotedblleft co-eventual family\textquotedblright\textbf{. }We define
accumulation points of eventual families in a Hausdorff topological space and
define the \textquotedblleft image family\textquotedblright\ $\mathcal{G}$ of
an eventual family $\mathcal{F}$ under a given mapping $f,$ called
\textquotedblleft the push of $\mathcal{F}$ by $f$\textquotedblright\ via
$\mathcal{G=}\mathrm{Push}(f,\mathcal{F)}:=\{S\subseteq Y\,\mid f^{-1}%
(S)\in\mathcal{F}\}.$

Focusing on eventual families in the set $\mathbb{N}$ of the integers enables
us to talk about sequences of points, particularly, points that are generated
by repeated application of an operator $T:X\rightarrow X.$ We then define the
notion of an \textquotedblleft\textbf{$\mathcal{E}$}-limit of a sequence
$(A_{n})_{n\in\mathbb{N}}$ of subsets\ of a set $X$\textquotedblright\ as the
set of all $x\in X$ such that the set of $n$ with $x\in A_{n}$ belongs to
\textbf{$\mathcal{E},$} i.e., $\mathbf{\mathcal{E}}$-$\lim_{n\rightarrow
\infty}A_{n}:=\{x\in X\mid\{n\,|\,x\in A_{n}\}\in\mathbf{\mathcal{E}}\}$ where
$\mathcal{E}$ is an eventual family in $\mathbb{N}.$ The relationship of this
notion with the classical notion of limit of a sequence of sets is studied.

In the sequel we expand our work to the notion of a \textquotedblleft
multiset\textquotedblright\ which is a modification of the concept of a set
that allows for multiple instances of its elements. The number of instances
given for each element is called the multiplicity of that element in the
multiset. With multisets in hand we define and develop \textquotedblleft
multifamilies\textquotedblright\ which are either \textquotedblleft
increasing\textquotedblright\ or \textquotedblleft
decreasing\textquotedblright, connecting with the earlier notions via the
statement that a family of subsets of $X$ is an eventual (resp.\ co-eventual)
family if the multifamily that defines it is increasing (resp.\ decreasing).

The abstract structure created here is motivated by, and feeds back to, our
look at the convergence analysis of an iterative process for asymptotically
finding a common fixed point of a family of operators. This particular case
serves as an example of the possible use of our theory. The work presented
here adds a new angle to the theory of set convergence, see, e.g., the books
by Rockafellar and R.J.-B. Wets \cite[Chapter 4]{rock-book} and by Burachik
and Iusem \cite{burachik-book}.

\section{Eventual Families and Their Use in Limiting Processes\label{sect:Ev}}

\subsection{Eventual Families\label{subsec:eve-fams}}

We introduce the following notion of eventual families of subsets.

\begin{definition}
\label{def:event-co-event}Let $X$ be a set and let $\mathcal{F}$ be a family
of subsets of $X.$ The family $\mathcal{F}$ is called an \textquotedblleft%
\textbf{eventual family\textquotedblright} if it is \textit{upper hereditary
with respect to inclusion}, i.e.,\ if%
\begin{equation}
S\in\mathcal{F},\,S^{\prime}\supseteq S\,\Rightarrow S^{\prime}\in\mathcal{F}.
\end{equation}
The family $\mathcal{F}$ is called a \textquotedblleft\textbf{co-eventual
family\textquotedblright\ }if it is \textit{lower hereditary with respect to
inclusion}, i.e.,\ if%
\begin{equation}
S\in\mathcal{F},\,S^{\prime}\subseteq S\Rightarrow S^{\prime}\in\mathcal{F}.
\end{equation}

\end{definition}

We mention in passing that Borg \cite{borg-hereditary} uses the term
\textquotedblleft hereditary family\textquotedblright, in his work in the area
of combinatorics, for exactly what we call here \textquotedblleft co-eventual
family\textquotedblright. Several simple observations regarding such families
can be made.

\begin{proposition}
\label{lem:simple-observ}(i) A family $\mathcal{F}$ of subsets of $X$ is
co-eventual iff its complement, i.e., the family of subsets of $X$ which are
not in $\mathcal{F}$, is eventual.

(ii) The empty family and the family of all subsets of $X$ are each both
eventual and co-eventual, and they are the only families with this property.
\end{proposition}

\begin{proof}
(i) This follows from the definitions. (ii) That the empty family and the
family of all subsets of $X$ are each both eventual and co-eventual is
trivially true. We show that if $\mathcal{F}$ is eventual and co-eventual and
is nonempty then it must contain all subsets of $X.$ Let $S\in\mathcal{F}$ and
distinguish between two cases. If $S=\emptyset$ then\thinspace\thinspace
$\mathcal{F}$ must contain all subsets of $X$ because $\mathcal{F}$ is
eventual. If $S\neq\emptyset$ let $x\in S$, then, since $\mathcal{F}$ is
co-eventual it must contain the singleton $\{x\}$. Consequently, the set
$\{x,y\},$ for any $y,$ is also in $\mathcal{F}$ and so $\{y\}\in\mathcal{F}$,
thus, all subsets of $X$ are contained in $\mathcal{F}$. Alternatively, if we
look at $S\in\mathcal{F}$, then for any subset $S^{\prime}$ of $X$,
$\mathcal{F}$ contains $S\cup S^{\prime}$ since $\mathcal{F}$ is eventual.
Then since $\mathcal{F}$ is co-eventual, it must contain $S^{\prime}$, leading
to the conclusion that it contains all subsets.
\end{proof}

\begin{remark}
\label{rem:filter}An eventual family $\mathcal{F}$ need not contain the
intersection of two of its members. If it does so for every two of its members
then it is a \textit{filter}.
\end{remark}

Similar to the notion used in \cite{danzig-folkman-shapiro} and
\cite{lent-censor} in the finite-dimensional space setting, we make here the
next definition.

\begin{definition}
\label{def:star-set}Given a family $\mathcal{F}$ of subsets of a set $X$, the
\textquotedblleft\textbf{star set associated with} $\mathcal{F}$%
\textquotedblright, denoted by\thinspace\thinspace$\mathrm{Star}%
(\mathcal{F}),$ is the subset of $X$ that consists of \textit{all }$x\in X$
\textit{such that the singletons }$\{x\}\in\mathcal{F}$, namely,%
\begin{equation}
\mathrm{Star}(\mathcal{F}):=\{x\in X\mid\{x\}\in\mathcal{F}\}.
\end{equation}

\end{definition}

\subsection{Accumulation Points as Limits of Eventual
Families\label{subsect:LimEv}}

Suppose now that $X$ is a \textit{Hausdorff }Topological space.

\begin{definition}
\label{def:limit-pt-and-set}Let $\mathcal{F}$ be an eventual family of subsets
of $X$. A point $x\in X$ is called an \textquotedblleft\textbf{accumulation
(or limit) point} of $\mathcal{F}$\textquotedblright\ if every (open)
neighborhood \footnote{Since, by definition, a neighborhood always contains an
\textit{open} neighborhood, considering all neighborhoods or just the open
ones does not make a difference here.} of $x$ belongs to $\mathcal{F}$. The
set of all accumulation points of $\mathcal{F}$ is called the
\textquotedblleft\textbf{limit set} of $\mathcal{F}$\textquotedblright.
\end{definition}

\begin{proposition}
\label{prop:lim-set-closed}The limit set of an eventual family $\mathcal{F}$
is always closed.
\end{proposition}

\begin{proof}
We show that the complement of the limit set, i.e., the set of all
non-accumulation points, is open. The point $y$ is a non-accumulation point
iff it has an open neighborhood which does not belong to $\mathcal{F}$,
i.e.,\ when it is a member of some open set not in $\mathcal{F}$. Hence the
complement of the limit set is the union of all open sets not in $\mathcal{F}%
$, and by definition, in a topological space, the union of any family of open
sets is open.
\end{proof}

We turn our attention now to sequences in $X$, i.e.,\ maps $\mathbb{N}%
\rightarrow X,$ where $\mathbb{N}$ denotes the positive integers.

\begin{definition}
\label{def:push}Given are a family $\mathcal{F}$ of subsets of $X$\ and a
mapping between sets $f:X\rightarrow Y$. The family $\mathcal{G}$ of subsets
of $Y$ whose inverse image sets $f^{-1}(S)$ belong to $\mathcal{F}$ will be
denoted by $\mathcal{G=}\mathrm{Push}(f,\mathcal{F)}$ and called the
\textquotedblleft\textbf{push} of $\mathcal{F}$ by $f$\textquotedblright$,$
namely,%
\begin{equation}
\mathcal{G=}\mathrm{Push}(f,\mathcal{F)}:=\{S\subseteq Y\,\mid f^{-1}%
(S)\in\mathcal{F}\}.
\end{equation}

\end{definition}

Combining Definitions \ref{def:limit-pt-and-set} and \ref{def:push} the
following remark is obtained.

\begin{remark}
\label{claim:push}Let $\mathcal{E}$ be an eventual family of subsets
\textit{of }$\mathbb{N}$ and let $f:\mathbb{N}\rightarrow X$ be defined by
some given sequence $(x_{n})_{n\in\mathbb{N}}$ in $X$. The accumulation points
and the limit set of $(x_{n})_{n\in\mathbb{N}}$ with respect to $\mathcal{E}$
are those defined with respect to the push of $\mathcal{E}$ by $f$ .
\end{remark}

The next examples emerge by using two different eventual families in
$\mathbb{N}$. The same `machinery' yields both `cases' via changing the
eventual family $\mathcal{E}$ in $\mathbb{N}$.

\begin{Exmps}
\label{exmps}

\begin{enumerate}
\item Let $\mathcal{E}$ be the family of complements of finite sets in
$\mathbb{N}$. Then accumulation points (i.e., limits with respect to
$\mathcal{E}$) are the usual limits, and if there is a limit point then it is
unique. This is the case, as one clearly sees, in a Hausdorff space $X$
whenever $\mathcal{E}$ is a filter, as here $\mathcal{E}$ clearly is.

\item Let $\mathcal{E}$ be the family of infinite subsets in $\mathbb{N}$.
Then being an accumulation point means \textit{being some accumulation point
of the sequence} in the usual sense, which in general, need not be unique.
Indeed, here $\mathcal{E}$ is not a filter.
\end{enumerate}
\end{Exmps}

\subsection{Operators and Seeking Fixed Points\label{subsect:OpFxPt}}

Continuing to consider a Hausdorff topological space $X$, call any continuous
self-mapping $T:X\rightarrow X$ \textquotedblleft\textbf{an operator}%
\textquotedblright.

\begin{definition}
\label{def:seq-follow-T}Let $X$ be a Hausdorff topological space,
$T:X\rightarrow X$ an operator, $(x_{n})_{n\in\mathbb{N}}$ a sequence in $X,$
and $\mathcal{E}$ an eventual family of subsets of $\mathbb{N}$. We say that
\textquotedblleft\textbf{the sequence }$(x_{n})_{n\in\mathbb{N}}$\textbf{
follows }$T$\textbf{ with respect to $\mathcal{E}$\textquotedblright} if, for
every $S\in\mathcal{E}$, there are integers $p,q$ in $S$ so that
$x_{p}=T(x_{q}).$
\end{definition}

\begin{theorem}
\label{tm} In a Hausdorff topological space $X$, if a sequence $(x_{n}%
)_{n\in\mathbb{N}}$ follows a continuous operator $T$ with respect to some
eventual family $\mathcal{E}$ in $\mathbb{N}$, and if $y$ is an accumulation
point of the sequence with respect to $\mathcal{E}$ then $y$ is a fixed point
of $T$.
\end{theorem}

\begin{proof}
Assume to the contrary that\ $T(y)\neq y$. Then, since the space is Hausdorff,
$T(y)$ and $y$ have disjoint open neighborhoods $U_{y}$ and $U_{T(y)}$.
Continuity of $T$ guarantees that there is an open neighborhood $V_{y}$ of $y$
so that $T(V_{y})\subset U_{T(y)}$. Hence,%
\begin{equation}
U_{y}\cap T(V_{y})=\emptyset, \label{eq:disjoint}%
\end{equation}
meaning that $T(z)\neq z$ for $z\in U_{y}\cap V_{y}$. But $U_{y}\cap V_{y}$ is
also an open neighborhood of $y$, and $y$ is an accumulation point of the
sequence with respect to $\mathcal{E}$, hence, the set%
\begin{equation}
S:=\{n\in\mathbb{N}\mid\,x_{n}\in U_{y}\cap V_{y}\}
\end{equation}
is in $\mathcal{E}$. Since the sequence follows $T$ with respect to
$\mathcal{E}$, there must be $p$ and $q$ in $S$ so that $x_{p}=T(x_{q})$. This
point must belong to both $U_{y}$ and $T(V_{y})$, which contradicts
(\ref{eq:disjoint}).
\end{proof}

\subsection{Finitely-Insensitive Eventual Families in $\mathbb{N}%
$\label{subsec:FinIns}}

When considering eventual families in $\mathbb{N}$ it is often desirable to
assume that they are \textit{finitely-insensitive}, as we define next. All our
examples have this property.

\begin{definition}
\label{def:finite-intensive}A family $\mathcal{E}$ of subsets of $\mathbb{N}$
is called a \textquotedblleft\textbf{finitely-insensitive
family\textquotedblright} if for any $S\in\mathcal{E}$, finitely changing $S$,
which means here\ adding and/or deleting a finite number of its members, will
result in a set $S^{\prime}\in\mathcal{E}$.
\end{definition}

\subsection{Limits of Sequences of Sets\label{subsect:LimSeqSt}}

In \cite{lent-censor}, \cite{danzig-folkman-shapiro} and \cite{salinetti-wets}
the notions of \textit{upper limit} and \textit{lower limit} of a sequence of
subsets $(A_{n})_{n\in\mathbb{N}}$ of some $X$ are considered, in the
framework of the Euclidean space, a locally compact metric space, or a normed
linear space of finite dimension, respectively. When these upper limit
$\limsup_{n\rightarrow\infty}A_{n}$ and lower limit $\liminf_{n\rightarrow
\infty}A_{n}$ coincide one says that the sequence of sets has their common
value as a \textit{limit}, denoted by $\lim_{n\rightarrow\infty}A_{n}.$ Thus,
a function defined on sets, or taking values in sets, may be said to be
\textit{continuous} when it respects limits of sequences.

Here we define the notion of an \textquotedblleft\textbf{$\mathcal{E}$}-limit
of a sequence $(A_{n})_{n\in\mathbb{N}}$ of subsets\ of a set $X$%
\textquotedblright\ and state its relationship with the classical notion of
limit mentioned above.

\begin{definition}
\label{def:E-limit}Let $X$ be a set, let $(A_{n})_{n\in\mathbb{N}}$ be a
sequence of subsets of $X,$ let $\mathcal{E}$ be an eventual family in
$\mathbb{N}$ and assume that $\mathcal{E}$ is finitely-insensitive. The
\textquotedblleft\textbf{$\mathcal{E}$}-limit of the sequence $(A_{n}%
)_{n\in\mathbb{N}}$\textquotedblright, denoted by $\mathbf{\mathcal{E}}$%
-$\lim_{n\rightarrow\infty}A_{n},$ is the set of all $x\in X$ such that the
set of $n$ with $x\in A_{n}$ belongs to $\mathcal{E}$, namely,%
\begin{equation}
\mathbf{\mathcal{E}}\text{-}\lim_{n\rightarrow\infty}A_{n}:=\{x\in
X\mid\{n\,|\,x\in A_{n}\}\in\mathbf{\mathcal{E}}\}.
\end{equation}

\end{definition}

Strict logic tells us that the $\mathcal{E}$-limit is well-defined also for an
empty $\mathcal{E}$ or if $\mathcal{E}$ contains all subsets. Indeed, if
$\mathcal{E}=\emptyset$ then $\mathbf{\mathcal{E}}$-$\lim_{n\rightarrow\infty
}A_{n}=\emptyset,$ and if $\mathcal{E}$ is the family of all subsets then
$\mathbf{\mathcal{E}}$-$\lim_{n\rightarrow\infty}A_{n}=X$.

\begin{theorem}
\label{thm:limit-E-limit}Let $X$ be a set, let $(A_{n})_{n\in\mathbb{N}}$ be a
sequence of subsets of $X,$ and let $\mathcal{E}$ be an eventual family in
$\mathbb{N}$. If $\mathcal{E}$ is a finitely-insensitive family which is not
trivial, i.e.,\ is not either empty or containing all subsets, and if the
(classical) $\lim_{n\rightarrow\infty}A_{n}$ exists then%
\begin{equation}
\mathbf{\mathcal{E}}\text{-}\lim_{n\rightarrow\infty}A_{n}=\lim_{n\rightarrow
\infty}A_{n}.
\end{equation}

\end{theorem}

\begin{proof}
Note that, for a given sequence of sets $(A_{n})_{n\in\mathbb{N}}$, the
`larger' the eventual family $\mathcal{E}$ is, the `larger' is its
$\mathcal{E}$-limit.

Denote by $\mathcal{G}$ the family of all\textit{ infinite }subsets of
$\mathbb{N}$ and by $\mathcal{H}$ the family of all subsets of $\mathbb{N}%
$\textit{ }with \textit{finite complement\footnote{The families $\mathcal{G}$
and $\mathcal{H}$ were denoted by $\mathcal{N}_{\infty}^{\#}$ and
$\mathcal{N}_{\infty}$, respectively, in \cite[page 108]{rock-book}.}}. Then
clearly (cf.\ Examples \ref{exmps}) The upper limit (resp.\ lower limit) of
$A_{n}$ is obtained as $\mathcal{E}\text{-}\lim_{n\rightarrow\infty}A_{n}$ for
$\mathcal{E}:=\mathcal{G}$\thinspace(resp.\ $\mathcal{E}:=\mathcal{H}$.)

Now, The family $\mathcal{G}$ is the largest finitely-insensitive family which
is not the set of all subsets. This is so because if $\mathcal{G}$ would
contain a finite set then it would have to contain the empty set, hence, all subsets.

And the family $\mathcal{H}$ is the smallest finitely-insensitive family which
is not empty. This is so because if $\mathcal{H}$ is not empty, it has a
member $S$, thus, must contain the whole $\mathbb{N}$, {hence}, all subsets
with finite complement.

Consequently, for a sequence $(A_{n})_{n\in\mathbb{N}}$ for which
$\lim_{n\rightarrow\infty}A_{n}$ exists, that limit will be also the
$\mathcal{E}$-limit for any finitely-insensitive eventual family $\mathcal{E}$
which is not trivial, i.e.,\ is not either empty or containing all subsets.
\end{proof}

\subsection{Topological vs.\ Purely Set-Theoretical\label{subsec:topo-vs-set}%
\ }

Note that in contrast to Subsections \ref{subsect:LimEv} and
\ref{subsect:OpFxPt}, the notions in Subsection \ref{subsect:LimSeqSt} are
purely set-theoretic and do not involve any topology in $X$. Yet, one can
distill the topological aspect via the next definition.

\begin{definition}
\label{def:distill}Let $X$ be a Hausdorff topological space and let
$\mathcal{F}$ be an eventual family in $X$. The \textquotedblleft%
\textbf{closure} of an eventual family $\mathcal{F}$ in $X$\textquotedblright,
denoted by $\mathrm{cl}\mathcal{F}$, consists of \textit{all subsets
}$S\subseteq X$ \textit{such that all the open subsets }$U\subseteq X$
\textit{which contain }$S$ \textit{belong to $\mathcal{F}$.}
\end{definition}

Clearly, $\mathcal{F}$ is always a subfamily of $\mathrm{cl}\mathcal{F}$, and
the set of limit points of an eventual family $\mathcal{F}$, in a Hausdorff
topological space $X,$ is just $\mathrm{Star}(\mathrm{cl}\mathcal{F}),$ given
in Definition \ref{def:star-set}.

\section{Multisets and Multifamilies\label{sect:mult-fam}}

A \textbf{multiset} (sometimes termed \textbf{bag}, or \textbf{mset}) is a
modification of the concept of a set that allows for multiple instances for
each of its elements. The number of instances given for each element is called
the multiplicity of that element in the multiset. The multiplicities of
elements are any number in $\{0,1,\ldots,\infty\}$, see the corner-stone
review of Blizard \cite{blizard-multiset-1989}.

\begin{definition}
(i) A \textbf{multiset} $M$ in a set $X$ is represented by a function
$\varphi_{M}:X\rightarrow\{0,1,\ldots,\infty\}$ such that for any $x\in X,$
$\varphi_{M}(x)$ is the multiplicity of $x$ in $M$. We refer to this function
as the \textquotedblleft\textbf{representing function of the multiset}%
\textquotedblright. If $\varphi_{M}(x)=0$ then the multiplicity $0$ means `not
belonging to the set'. A subset $S\subseteq X$ is a multiset represented by
$\iota_{S}$, the \textquotedblleft\textit{indicator function\textquotedblright%
} of $S,$ i.e.,%
\begin{equation}
\iota_{S}(x):=\left\{
\begin{array}
[c]{cc}%
1, & \mathrm{if}\text{ \ }x\in S,\\
0, & \mathrm{if}\text{ \ }x\notin S.
\end{array}
\right.
\end{equation}

(ii) A \textbf{multifamily} $\mathcal{M}$ on a set $X$ is a multiset in the
powerset $2^{X}$ of $X$ (i.e., all the subsets of $X$). Its representing
function, denoted by $\varphi_{\mathcal{M}}:2^{X}\rightarrow\{0,1,\ldots
,\infty\},$is such that for any $S\subseteq X,$ $\varphi_{M}(S)$ is the
multiplicity of $S$ in $\mathcal{M}$. A family $\mathcal{F}$ of subsets of $X$
is a multifamily on $X$ represented by $\iota_{\mathcal{F}}$, the
\textquotedblleft\textit{indicator function\textquotedblright\ of
}$\mathcal{F}$, i.e.,%
\begin{equation}
\iota_{\mathcal{F}}(f):=\left\{
\begin{array}
[c]{cc}%
1, & \mathrm{if}\text{ \ }f\in\mathcal{F},\\
0, & \mathrm{if}\text{ \ }f\notin\mathcal{F}.
\end{array}
\right.
\end{equation}

(iii) A multifamily $\mathcal{M}$ on a set $X$ with a representing function
$\varphi_{\mathcal{M}}$is called \textbf{increasing} if%
\begin{equation}
S,S^{\prime}\subseteq X,\,S\subseteq S^{\prime}\Rightarrow\varphi
_{\mathcal{M}}(S)\leq\varphi_{\mathcal{M}}(S^{\prime}),
\end{equation}
and called \textbf{decreasing} if%
\begin{equation}
S,S^{\prime}\subseteq X,\,S\subseteq S^{\prime}\Rightarrow\varphi
_{\mathcal{M}}(S)\geq\varphi_{\mathcal{M}}(S^{\prime}).
\end{equation}

\end{definition}

Clearly, a\textit{ family }of subsets of $X$ is an \textit{eventual
(resp.\ co-eventual) family if the multifamily that defines it is increasing
(resp.\ decreasing)}. The next example shows why these notions may be useful.

\begin{example}
\label{ex:Gap} Considering the set $\mathbb{N}$, for a, finite or infinite,
subset $S\subseteq\mathbb{N}$ write $S$ as%
\begin{equation}
S=\{n_{1}^{S},n_{2}^{S},\ldots\},
\end{equation}
where $n_{\ell}^{S}\in\mathbb{N}$ for all $\ell,$ and the sequence $(n_{\ell
}^{S})_{\ell=1}^{L}$ (where $L$ is either finite or $\infty$) is strictly
increasing, i.e., $n_{1}^{S}<n_{2}^{S}<\ldots$. We consider the\textit{
}\textbf{gaps }between consecutive elements\textit{ }in\textit{ }$S$ as the
sequence of differences%
\begin{equation}
n_{2}^{S}-n_{1}^{S}-1,n_{3}^{S}-n_{2}^{S}-1,\ldots,
\end{equation}
where, if $S$ is finite add $\infty$ at the end. Defining%
\begin{equation}
\mathrm{Gap}(S):=\limsup_{k}(n_{k+1}^{S}-n_{k}^{S}-1),
\end{equation}
makes $\mathrm{Gap}$ a\textit{ }multifamily\textit{ }on\textit{ }$\mathbb{N}$,
thus taking values in $\{0,1,\ldots,\infty\}$, in particular, taking the value
$\infty$ for (among others) any finite $S$.
\end{example}

Note that if $\mathrm{Gap}(S)$ is finite then there must be an infinite number
of differences $(n_{k+1}^{S}-n_{k}^{S}-1)$ equal to $\mathrm{Gap}(S)$, but
this is not true for any larger integer - because by the definition of
$\limsup$ and because we are dealing with integer-valued items, a finite
$\limsup$ must actually be attained an infinite number of times.

Observe further that the larger the set $S$ is -- the smaller (or equal) is
$\mathrm{Gap}(S).$ Thus, $\mathrm{Gap}$ is a decreasing multifamily.

Define the complement-multifamily for some multifamily $\mathcal{G}$ on the
subsets of a set $X$ by%
\begin{equation}
\mathcal{G}^{c}(S):=\mathcal{G}(S^{c}),\quad\forall S\subseteq X
\label{eq:script-G-complement}%
\end{equation}

where $S^{c}$ is the complement of $S$ in $X$.

We will focus on $\mathrm{coGap}:=\mathrm{Gap}^{c}$. For any $S\subseteq
\mathbb{N},$ let us denote by $c_{S}$ the maximal number of integers between
consecutive elements of $S,$ namely, between $n_{\ell}^{S}\in S$ and
$n_{\ell+1}^{S}\in S.$ If $S$ has arbitrarily big such `intervals' between
consecutive elements then we write $c_{S}=\infty$. With this in mind,
$\mathrm{coGap}=\mathrm{Gap}^{c}$ is an increasing multifamil\textit{y} equal
to $(c_{S})_{\forall S\subseteq\mathbb{N}}.$

\subsection{Extensions to Multifamilies\label{subsect:transferring}}

We now extend some of the notions of Subsection \ref{subsec:eve-fams} to multifamilies.

\begin{definition}
\label{def:star-set copy(1)}Given a multifamily $\mathcal{M}$ on the subsets
of a set $X$ whose representing function is $\varphi_{\mathcal{M}}.$ The
\textquotedblleft\textbf{star set associated with} $\mathcal{M}$%
\textquotedblright, denoted by\thinspace\thinspace$\mathrm{Star}%
(\mathcal{M}),$ is the multiset $M$ on $X$ whose representing function
$\varphi_{M}$ is related to $\varphi_{\mathcal{M}}$ in the following manner%
\begin{equation}
\mathrm{Star}(\mathcal{M}):=M,\text{ such that }\varphi_{M}(x)=\varphi
_{\mathcal{M}}(\{x\}).
\end{equation}

\end{definition}

\begin{definition}
\label{def:push copy(1)}Given a multifamily $\mathcal{M}$ on the subsets of
$X$\ whose representing function is $\varphi_{\mathcal{M}}$ and a mapping
between sets $f:X\rightarrow Y$. The multifamily $\mathcal{G}$ on the subsets
of $Y,$ denoted by $\mathcal{G=}\mathrm{Push}(f,\mathcal{M)}$, with
representing function $\varphi_{\mathcal{G}},$ will be called the
\textquotedblleft\textbf{push} of $\mathcal{M}$ by $f$\textquotedblright\ if
its representing function is related to the representing function of
$\mathcal{M}$ in the following manner {%
\begin{equation}
\mathcal{G=}\mathrm{Push}(f,\mathcal{M)}\text{ such that }\varphi
_{\mathcal{G}}(S)=\varphi_{\mathcal{M}}(f^{-1}(S)).
\end{equation}
}
\end{definition}

\begin{definition}
\label{def:finite-intensive copy(1)}A multifamily $\mathcal{M}$ of subsets of
$\mathbb{N}$ whose representing function is $\varphi_{\mathcal{M}}$ is called
a \textquotedblleft\textbf{finitely-insensitive multifamily\textquotedblright}
if for any $S\in\mathcal{M}$, finitely changing $S$, i.e.,\ adding and/or
deleting a finite number of its members, will not change its multiplicity,
i.e., will result in a set $S^{\prime}\in\mathcal{M}$ such that $\varphi
_{\mathcal{M}}(S)=\varphi_{\mathcal{M}}(S^{\prime})$.
\end{definition}

\begin{definition}
\label{def:distill copy(1)}Let $X$ be a Hausdorff topological space and let
$\mathcal{M}$ be an \textit{increasing} multifamily whose representing
function is $\varphi_{\mathcal{M}}$. The \textquotedblleft\textbf{closure of
an increasing multifamily }$M$\textbf{ in }$X$\textquotedblright, denoted by
$\mathrm{cl}\mathcal{M}$, is defined to be the (increasing) multifamily such
that for any $S\subseteq X$ it holds that%
\begin{equation}
\varphi_{\mathrm{cl}\mathcal{M}}(S)=\min\{\varphi_{\mathcal{M}}(U)\mid\text{
all\ \textit{open subsets }}U\subseteq X\text{ such that }S\subseteq
U\}.\text{ }%
\end{equation}

\end{definition}

\begin{definition}
\label{def:it:lim}Let $X$ be a Hausdorff topological space and let
$\mathcal{M}$ be an \textit{increasing} multifamily whose representing
function is $\varphi_{\mathcal{M}}$. The multiset $M:=\mathrm{Star}%
(\mathrm{cl}\mathcal{M})$ will be called the \textquotedblleft%
\textbf{multiset-limit} of $\mathcal{M}$\textquotedblright\ and denoted by
$\lim\mathcal{M}$. Its representing function is for any $x\in X,$%
\begin{equation}
\varphi_{M}(x)=\min\{\varphi_{\mathcal{M}}(U)\mid\text{all\ open
subsets}\mathit{\ }U\subseteq X\text{ such that }x\in U\}.
\end{equation}

\end{definition}

Given a multifamily $\mathcal{M}$ on the subsets of $\mathbb{N}$ whose
representing function is $\varphi_{\mathcal{M}}$, the `limiting notions' with
respect to $\mathcal{M}$ for a sequence $(x_{n})_{n\in\mathbb{N}}$, are
defined as those with respect to $\mathrm{Push}(f,\mathcal{M)}$ of
$\mathcal{M}$ to $X$ by the function $f$ $f:\mathbb{N}\rightarrow X$ which
represents the sequence $(x_{n})_{\ }$. In particular, for an increasing
multifamily $\mathcal{M}$ on the subsets of $\mathbb{N}$ whose representing
function is $\varphi_{\mathcal{M}}$, the multiset limit of $\mathrm{Push}%
(f,\mathcal{M)}$ will be called the \textquotedblleft\textbf{multiset-limit}
of $(x_{n})$\textquotedblright, denoted by $\lim_{\mathcal{M}}x_{n}.$

Denoting the representing function of this multiset $\mathcal{G}$ on $X$ by
$\varphi_{\mathcal{G}},$ we can describe it as follows. Given a point $x\in
X,$ consider the following subsets of $\mathbb{N}${%
\begin{equation}
S(U):=\{n\in\mathbb{N}\mid x_{n}\in U\},\text{ for open neighborhoods }U\text{
of }x.
\end{equation}
}

Then,%
\begin{equation}
\varphi_{\mathcal{G}}(x)=\min\{\varphi_{\mathcal{M}}(S(U))\mid\text{all\ open
subsets}\mathit{\ }U\subseteq X\text{ such that }x\in U\}.
\end{equation}

\begin{remark}
\label{rerere} Note, that for a set $S$ not to belong to $\mathrm{coGap}$,
i.e.,\ to have $\mathrm{coGap}(S)=0,$ just means that $S$ is finite - as a
`family, ignoring multiplicities' and $\mathrm{coGap}$ is just the family of
\textit{infinite} sets of natural\ numbers.

Thus, when we turn to the \textit{limit} of a sequence $(x_{n})_{n\in
\mathbb{N}}$ in a Hausdorff Space $X$ (a notion which is obviously dependent
on the topology. In a Banach or Hilbert space we will have strong and weak
limits etc.); and we take the $\mathrm{coGap}$-limit (it will be a multiset on
$X$, to which for some $x$ in $X$ to belong (at least) $n$ times, one must
have, for every neighborhood $U$ of $x$, that the $x_{n}$ stay in $U$ for some
$n$ consecutive places as far as we go); then the $\mathrm{coGap}$-limit of
$(x_{n})_{n\in\mathbb{N}}$, `forgetting the multiplicities' is just the set of
accumulation points of $(x_{n})_{n\in\mathbb{N}}$ (which is, recalling the
examples in Section \ref{sect:Ev}, just its $\mathcal{G}$-limit for
$\mathcal{G}$ the eventual family of the infinite subsets of $\mathbb{N}$).

Note that, in general, if the sequence has a limit $x^{\ast}$ (in the good old
sense) then its $\mathrm{coGap}$-limit `includes $x^{\ast}$ infinitely many
times and does not include any other point'. This sort of indicates to what
extent the $\mathrm{coGap}$-limit may be viewed as `more relaxed' than the
usual limit.

The inverse implication does not always hold (it holds however in a compact
space) as the following counterexample shows. In $\mathbb{R}$ (the reals),
define a sequence by%
\begin{equation}
x_{2n}:=n\text{ \ and \ }x_{2n-1}:=-1
\end{equation}
then its $\mathrm{coGap}$-limit contains $-1$ infinitely often and does not
contain others, but $-1$ is not a limit.
\end{remark}

\section{Convergence of Algorithms for Solving the Common Fixed-Point Problem}

Given a finite family of self-mapping operators $\left\{  T_{i}\right\}
_{i=1}^{m}$ acting on the Hilbert space $H$ with $\operatorname*{Fix}T_{i}%
\neq\emptyset,$ $i=1,2,\ldots,m,$ where $\operatorname*{Fix}T_{i}:=\{x\in
H\mid T_{i}(x)=x\}$ is the fixed points set of $T_{i},$ the \textquotedblleft%
\textbf{common fixed point problem}\textquotedblright\ (CFPP) is to find a
point%
\begin{equation}
x^{\ast}\in\cap_{i=1}^{m}\operatorname*{Fix}T_{i}. \label{Common fixed pp}%
\end{equation}
This problem serves as a framework for handling many important aspects of
solving systems of nonlinear equations, feasibility-seeking of systems of
constraint sets and optimization problems, see, e.g., the excellent books by
Berinde \cite{Berinde-book} and by Cegielski \cite{Cegielski-book} and
references therein. In particular, iterative algorithms for the CFPP form an
ever growing part of the field. There are many algorithms around for solving
CFPPs, see, e.g., Zaslavski's book \cite{zaslavski-book}. To be specific, we
use the \textquotedblleft Almost Cyclic Sequential Algorithm (ACSA) for the
common fixed-point problem\textquotedblright, which is Algorithm 5 in Censor
and Segal \cite{censor-segal-2009}, which is, in turn, a special case of an
algorithm in the paper by Combettes \cite[Algorithm 6.1]{Combettes01}. The
abstract study of limits of eventual families developed here can serve as a
unifying convergence analysis of many iterative processes. It grew out of our
look at the almost cyclic sequential algorithm and, therefore, we describe
this algorithm and its relation with the present work next.

\subsection{The Almost Cyclic Sequential Algorithm (ACSA)}

Let $\left\langle x,y\right\rangle $ and $\left\Vert x\right\Vert $ be the
Euclidean inner product and norm, respectively, in the $J$-dimensional
Euclidean space $R^{J}$. Given $x,y\in R^{J}$ we denote the half-space%
\begin{equation}
H(x,y):=\left\{  u\in R^{J}\mid\left\langle u-y,x-y\right\rangle
\leq0\right\}  .
\end{equation}

\begin{definition}
\label{Def of directed ops}An operator $T:R^{J}\rightarrow R^{J}$ is called
\textquotedblleft\textbf{a cutter\textquotedblright} if%
\begin{equation}
\operatorname*{Fix}T\subseteq H(x,T(x)),\text{ for all }x\in R^{J},
\end{equation}
or, equivalently,%
\begin{equation}
\text{if }z\in\operatorname*{Fix}T\text{ then }\left\langle T\left(  x\right)
-x,T\left(  x\right)  -z\right\rangle \leq0,\text{ for all }x\in R^{J}.
\label{def directed2}%
\end{equation}

\end{definition}

The class of cutters was called $\Im$-class by Bauschke and Combettes
\cite{BC01} who first defined this notion and showed (see \cite[Proposition
2.4]{BC01}) (i) that the set of all fixed points of a cutter $T$ with nonempty
$\operatorname*{Fix}T$ is closed and convex because%
\begin{equation}
\operatorname*{Fix}T=\cap_{x\in R^{J}}H\left(  x,T\left(  x\right)  \right)  ,
\end{equation}
and (ii) that the following holds%
\begin{equation}
\text{If }T\in\Im\text{ then }Id+\lambda(T-Id)\in\Im,\text{ for all }%
\lambda\in\lbrack0,1], \label{BCresult}%
\end{equation}
where $Id$ is the identity operator. This class of operators includes, among
others, the resolvents of a maximal monotone operators, the firmly
nonexpansive operators, namely, operators $N:R^{J}\rightarrow R^{J}$ that
fulfil%
\begin{equation}
\left\Vert N(x)-N(y)\right\Vert ^{2}\leq\left\langle
N(x)-N(y),x-y\right\rangle ,\text{ for all }x,y\in R^{J},
\end{equation}
the orthogonal projections and the subgradient projectors. Note that every
cutter belongs to the class of operators $\mathcal{F}^{0},$ defined by Crombez
\cite[p. 161]{Crombez05}. The term \textquotedblleft cutter\textquotedblright%
\ was proposed in \cite{ceg-cen-2011}, see \cite[pp. 53--54]{Cegielski-book}
for other terms that are used for these operators.

The following definition of a demiclosed operator that originated in Browder
\cite{Browder} (see, e.g., \cite{Combettes01}) will be required.

\begin{definition}
An operator $T:R^{J}\rightarrow R^{J}$ is said to be \textquotedblleft%
\textbf{demiclosed} \textbf{at }$y\in R^{J}$\textquotedblright\ if for every
$\overline{x}\in R^{J}$ and every sequence $(x_{n})_{n\in\mathbb{N}}$ in
$R^{J},$ such that, $\lim_{n\rightarrow\infty}x_{n}=\overline{x}$ and
$\lim_{n\rightarrow\infty}T(x_{n})=y,$ we have $T(\overline{x})=y.$
\end{definition}

For instance, the orthogonal projection onto a closed convex set is everywhere
a demiclosed operator, due to its continuity.

\begin{remark}
\cite{Combettes01} If $T:R^{J}\rightarrow R^{J}$ is nonexpansive, then $T-Id$
is demiclosed on $R^{J}.$
\end{remark}

In sequential algorithms for solving the common fixed point problem the order
by which the operators are chosen for the iterations is given by a
\textquotedblleft\textbf{control sequence\textquotedblright}\textit{ }of
indices\textit{\ }$(i(n))_{n\in\mathbb{N}},$ see, e.g., \cite[Definition
5.1.1]{Censor book}.

\begin{definition}
(i) \textbf{Cyclic control.} A control sequence is \textquotedblleft%
\textbf{cyclic\textquotedblright} if $i(n)=n\operatorname*{mod}m+1,$ where $m$
is the number of operators in the common fixed point problem.

(ii) \textbf{Almost cyclic control. }$(i(n))_{n\in\mathbb{N}}$ is
\textquotedblleft\textbf{almost cyclic on }$\{1,2,\ldots,m\}$%
\textquotedblright\ if $1\leq i(n)\leq m$ for all $n\geq0,$ and there exists
an integer $c\geq m$ (called the \textquotedblleft\textbf{almost cyclicality
constant\textquotedblright}), such that, for all $n\geq0$, $\{1,2,\ldots
,m\}\subseteq\{i(n+1),i(n+2),\ldots,i(n+c)\}.$
\end{definition}

Consider a finite family $T_{i}:R^{J}\rightarrow R^{J},$ $i=1,2,\ldots,m,$ of
cutters with $\cap_{i=1}^{m}\operatorname*{Fix}T_{i}\neq\emptyset$. The
following algorithm for finding a common fixed point of such a family is a
special case of \cite[Algorithm 6.1]{Combettes01}.

\begin{algorithm}
\textbf{Almost Cyclic Sequential Algorithm (ACSA) for solving common fixed
point problems }\cite[Algorithm 5]{censor-segal-2009}
\label{alg CFP of AV ops}$\left.  {}\right.  $

\textbf{Initialization:} $x_{0}\in R^{J}$ is an arbitrary starting point.

\textbf{Iterative Step: }Given\textbf{\ }$x_{n},$ compute $x_{n+1}$ by%
\begin{equation}
x_{n+1}=x_{n}+\lambda_{n}(T_{i(n)}\left(  x_{n}\right)  -x_{n}).
\label{eq. Iter of AVoperators}%
\end{equation}

\textbf{Control: }$(i(n))_{n\in\mathbb{N}}$ is almost cyclic on $\{1,2,\ldots
,m\}$.

\textbf{Relaxation parameters: }$(\lambda_{n})_{n\in\mathbb{N}}$ are confined
to the interval $\left[  0,2\right]  $.
\end{algorithm}

The convergence theorem of Algorithm \ref{alg CFP of AV ops} is as follows.

\begin{theorem}
\label{Theor. GenKM}Let $\left\{  T_{i}\right\}  _{i=1}^{m}$ be a finite
family of cutters $T_{i}:R^{J}\rightarrow R^{J}$, which satisfies

(i) $\Omega:=\cap_{i=1}^{m}\operatorname*{Fix}T_{i}$ is nonempty, and

(ii) $T_{i}-Id$ are demiclosed at $0,$ for every $i\in\{1,2,\ldots,m\}.$

Then any sequence $(x_{n})_{n\in\mathbb{N}},$ generated by Algorithm
\ref{alg CFP of AV ops}, converges to a point in $\Omega.$
\end{theorem}

\begin{proof}
This follows as a special case of \cite[Theorem 6.6 (i)]{Combettes01}.
\end{proof}

\subsection{An Abstract Approach to The Convergence of the ACSA}

Given a sequence $(x_{n})_{n\in\mathbb{N}}$ in a Hausdorff topological space
$X$, push the multiset $\mathrm{coGap}$ in $\mathbb{N}$ to a multiset
$\mathcal{M}$ on the subsets of $X$, and then consider its \textit{limit} $L$
(see Definitions \ref{def:distill copy(1)} and \ref{def:it:lim} above) with
respect to the multiset $\mathrm{Star}(\mathrm{cl}\mathcal{M})$ whose
representing function value at $x\in X$ is the minimum of the value of
$\mathrm{coGap}$ on the sets $\{n\in\mathbb{N}\,\mid\,x_{n}\in U\}$ for (open)
neighborhoods $U$ of $x$.

Then, by what was said in Subsection \ref{subsect:OpFxPt}, Example
\ref{ex:Gap} and Theorem \ref{tm}, we reach the following conclusion.

\begin{conclusion}
For an operator (i.e.,\ a continuous mapping) $T:X\rightarrow X$, if
$(x_{n})_{n\in\mathbb{N}}$ follows $T$ for the eventual family which is the
level family, for some $c,$%
\begin{equation}
\mathrm{coGap}_{c}:=\{S\subset\mathbb{N}\,\mid\mathrm{coGap}(S)\geq c\},
\end{equation}
then the level set $\{x\in X\mid\,L(x)\geq c\},$ where $L$ is the limit of the
multiset $\mathcal{M}$ on the subsets of $X$, mentioned above, will consist of
fixed points of $T$.
\end{conclusion}

This is the case with respect to each of the operators of the CFPP, for any
sequence generated by the ACSA. Thus, any sequence of iterations of the ACSA
follows each of the operators of the CFPP with respect to the eventual family
$\mathcal{E}_{c}$ in $\mathbb{N}$ consisting of \textit{all subsets of
}$\mathbb{N}$ \textit{that, after any number }$\mathit{N}$\textit{, contain
some `interval' of $c$ consecutive numbers} for some fixed number $c$.

This means that the eventual family $\mathcal{E}_{c},$ mentioned in Subsection
\ref{subsect:OpFxPt} as relevant to the sequence of iterations in the ACSA
will be just \textit{the `level family' }$\mathit{\{S\subset\mathbb{N}%
\mid\mathrm{coGap}(S)\geq c\}}$, and clearly any such level family of an
\textit{increasing} multiset is automatically an eventual family.\bigskip

\textbf{Declarations:}

Funding: T\textbf{he work of Yair Censor is supported by the Israel Science
Foundation and the Natural Science Foundation China, ISF-NSFC joint research
program Grant No. 2874/19.} (information that explains whether and by whom the
research was supported)

Conflicts of interest/Competing interests: \textbf{The authors have no
conflicts of interest to declare that are relevant to the content of this
article.} (include appropriate disclosures)

Availability of data and material: \textbf{Not applicable.} (data transparency)

Code availability: \textbf{Not applicable.} (software application or custom code)

Authors' contributions: \textbf{Not applicable.} (optional: please review the
submission guidelines from the journal whether statements are mandatory)

\end{document}